\documentclass[letterpaper,12pt]{article}

\usepackage[english]{babel}
\usepackage[utf8x]{inputenc}
\usepackage[T1]{fontenc}
\usepackage[affil-it]{authblk}

\usepackage[letterpaper,top=3cm,bottom=3cm,left=3cm,right=3cm,marginparwidth=1.75cm]{geometry}

\usepackage{amsmath}
\usepackage{amsthm}
\usepackage{amssymb}
\usepackage[all,cmtip]{xy}
\usepackage{graphicx}
\usepackage[colorinlistoftodos]{todonotes}
\usepackage[colorlinks=true, allcolors=blue]{hyperref}
\usepackage{tikz-cd}
\usepackage{theoremref}

\theoremstyle{plain}
\newtheorem{theorem}{Theorem}[section]
\newtheorem{lemma}[theorem]{Lemma}
\newtheorem{prop}[theorem]{Proposition}
\newtheorem{cor}[theorem]{Corollary}

\theoremstyle{definition}

\newtheorem{deff}[theorem]{Definition}
\usepackage{theoremref}

\newcommand{\Z}{\mathbb{Z}}
\newcommand{\Q}{\mathbb{Q}}

\newcommand{\C}{\mathbb{C}}
\newcommand{\BAG}{B^A(G)}
\newcommand{\BAH}{B^A(H)}
\newcommand{\gBAG}{\widetilde{B^A}(G)}
\newcommand{\gBAH}{\widetilde{B^A}(H)}
\newcommand{\MAG}{\mathcal{M}^A_G}
\newcommand{\MAH}{\mathcal{M}^A_H}
\newcommand{\SG}{\mathcal{S}_G}
\newcommand{\SH}{\mathcal{S}_H}

\title{Species isomorphisms of fibered Burnside rings}
\author{Benjam\'in Garc\'ia}
\affil{Centro de Ciencias Matem\'aticas,\\
Universidad Nacional Aut\'onoma de M\'exico\\ 
Morelia, Mich., M\'exico\\ 
benjamingarcia@matmor.unam.mx}

\begin{document}

\maketitle

\begin{abstract}
In this note, we present a notion of species isomorphism for fibered Burnside rings. We prove that a species isomorphism can be restricted to an isomorphism of mark tables, that it can be extended to an isomorphism between the ghost rings and we provide three additional characterizations of this concept.
\end{abstract}


\section{Introduction}

The \textit{Isomorphism Problem} of a representation ring is the question of whether two finite groups having isomorphic representation rings are isomorphic. This problem has been extensively studied for several representation rings and group algebras, with a negative answer in some cases, for example, there are pairs of non-isomorphic groups having isomorphic Burnside rings (see Kimmerle and Roggenkamp \cite{KimRog}, Raggi-Cárdenas and Valero-Elizondo \cite{RagVal1}, Thévenaz \cite{Thev}). 

On the other hand, the study of isomorphisms between representation rings often reveals common invariants preserved in the rings, in times leading to a positive answer to the isomorphism problem within certain classes of groups, e.g., if $G$ and $H$ are finite groups with isomorphic monomial representation rings, then if $G$ is nilpotent so is $H$, and if $G$ is either abelian or has square-free order, then $G\cong H$, while the general question remains open (see Müller \cite{Mull1,Mull2}). 

In the case of the Burnside ring, one feature that all known counterexamples share is that they come from isomorphisms of mark tables, which is equivalent to have an isomorphism betwen the Burnside rings preserving the standard bases. The latter kind of isomorphism has been considered by Raggi and Valero in \cite{RagValGlobal} for global representation rings, presented as the notion of \textit{species isomorphism}.

If $A$ is an abelian group, the \textit{$A$-fibered Burnside ring of a finite group} $G$, denoted by $B^A(G)$, is the Grothendieck group of the category of \textit{$A$-fibered $G$-sets}, introduced by Dress in \cite{Dress}. For the purposes of this note, we will work with a presentation of this ring in terms of a standard basis. The importance of $A$-fibered Burnside rings lies in their connection to the study of canonical induction formulae (see Boltje \cite{Bolt1,Bolt2}), and they are a source of examples in the theory of biset functors. One important case is $B^{\C^{\times}}(G)$, naturally isomorphic to the ring of monomial representations $D(G)$.

The aim of this note is to present a notion of \textit{species isomorphism} for fibered Burnside rings, as a natural analogy to isomorphisms of mark tables and species isomorphisms for global representation rings. In Section 3, we prove that a species isomorphism implies an isomorphism of mark tables, and that it can always be extended to an isomorphism between the ghost rings. The main result provides 3 characterizations of species isomorphisms that might be useful in future studies of the isomorphism problem.

\subsection*{Notation}

The letters $G$ and $H$ denote finite groups, $C_n$ stands for a cyclic group of order $n$, and we write $\mathcal{S}_G$ for the set of subgroups of $G$. If $K\leq G$ and $g,x\in G$, then $^gK=gKg^{-1}$, $K^g=g^{-1}Kg$, $^gx=gxg^{-1}$ and $x^g=g^{-1}xg$. If $G$ acts on a set $X$, then $G\cdot x$ stands for the orbit of $x\in X$ and $y=_G{x}$ means that $y\in G\cdot x$, $G\backslash X$ stands for the set of orbits and $[G\backslash X]$ is a set of representatives, $X^G$ denotes the set of fixed points. If $M$ is a free abelian group with a fixed basis $\mathfrak{B}$, then $\textit{coeff}_v(m)=a_v$ for $m\in M$ and $v\in \mathfrak{B}$ if $m=\sum_{v\in \mathfrak{B}}a_vv$, $a_v\in \Z$.

\section{Fibered Burnside rings}

Let $A$ be an abelian group and $G$ be a finite group, and set
$$\mathcal{M}^A_G=\left\{(K,\phi)|K\leq G, \phi\in Hom(K,A)\right\},$$
which is sometimes referred to as the \textit{set of subcharacters of G}. The group $G$ acts on $\MAG$ by $^g(K,\phi)=({^gK},{^g\phi})$ for $(K,\phi)\in \MAG$, $g\in G$, where $^g\phi(x)=\phi(x^g)$ for $x\in {^gK}$. We write $[K,\phi]_G$ for the $G$-orbit of a subcharacter $(K,\phi)$, and $N_G(K,\phi)$ for its stabilizer, and the contaiments $K\leq N_G(K,\phi)\leq N_G(K)$ are obvious. There is a partial order on $\mathcal{M}^A_G$ given by $(K,\phi)\leq (L,\psi)$ if $K\leq L$ and $\psi|_L=\phi$, turning $\mathcal{M}^A_G$ into a $G$-poset. This induces a partial order on $G\backslash \mathcal{M}^A_G$ by setting $[K,\phi]_G\leq [L,\psi]_G$ if there is $g\in G$ such that $(K,\phi)\leq {^g(L,\psi)}$. 

The \textit{$A$-fibered Burnside ring of $G$} is defined as the free abelian group
\begin{align}
    B^A(G)=\bigoplus_{[K,\phi]_G\in G\backslash\mathcal{M}^A_G}\Z [K,\phi]_G
\end{align}
with  multiplication given by
\begin{align}
    [K,\phi]_G\cdot [L,\psi]_G= \sum_{s\in [K\backslash G/L]}[K\cap {^sL}, \phi\cdot {^s\psi}]_G,
\end{align}
where for subgroups $K,L\leq G$ and morphisms $\phi:K\longrightarrow A$ and $\psi: L\longrightarrow A$, $\phi\cdot \psi=\phi|_{K\cap L}\psi|_{K\cap L}$, and $[K\backslash G/L]$ is a set of representatives of the double cosets $KsL$. Then $\BAG$ is a commutative associative ring with identity element $[G,1]_G$. 

It is straightforward that the subring generated by the elements $[K,1]_G$ is isomorphic to the \textit{Burnside ring} $B(G)$ if we identify $[K,1]_G$ with the isomorphim class of $G/K$. Note also that $\BAG\cong B(G)$ if and only if $A$ has trivial $|G|$-torsion: if $A$ has trivial $|G|$-torsion, then all the subcharacters of $G$ are trivial, and if there is $1\neq a\in A$ such that $a^{|G|}=1$, then there is a non-trivial subcharacter which cannot be $G$-conjugate to one of the form $(K,1)$, and so $\BAG$ has rank strictly greater than $B(G)$.

Expressing Equation (2) in terms of the basis $G\backslash \MAG$, we get
\begin{align}
    [K,\phi]_G\cdot [L,\psi]_G&=\sum_{[T,\tau]_G\in G\backslash \MAG}\mu^{G,(K,\phi)\times(L,\psi)}_{(T,\tau)} [T,\tau]_G
\end{align}
where
\begin{align}
    \mu^{G,(K,\phi)\times(L,\psi)}_{(T,\tau)}=|\{KsL\in K\backslash G/L\;|\;[T,\tau]_G&=[K\cap {^sL},\phi\cdot {^s\psi}]_G\}|.
\end{align}
These numbers will be referred to as the \textit{multiplication coefficients} of $\MAG$, and the following lemma summarizes some of their properties.

\begin{lemma}\thlabel{BurnsEls}
Let $(K,\phi), (L,\psi), (T,\tau)\in \MAG$. The following assertions hold:
\begin{enumerate}
    \item $\mu^{G,(^gK,^g\phi)\times(^xL,^x\psi)}_{(^yT,^y\tau)}=\mu^{G,(K,\phi)\times(L,\psi)}_{(T,\tau)}$ for all $g,x,y\in G$,
    \item $\mu^{G,(L,\psi)\times(K,\phi)}_{(T,\tau)}=\mu^{G,(K,\phi)\times(L,\psi)}_{(T,\tau)}$,
    \item $0\neq \mu^{G,(K,\phi)\times(L,\psi)}_{(T,\tau)}$ if and only if there is $g\in G$ such that $[T,\tau]_G=[K\cap {^gL},\phi\cdot {^g\psi}]_G$,
    \item $0\neq \mu^{G,(K,\phi)\times(K,\zeta)}_{(K,\zeta)}$ for every $\zeta\in Hom(K,A)$ if and only if $\phi =1$.
\end{enumerate}
\end{lemma}
\begin{proof}
Parts (1) - (3) follow from the definition. To prove Part (4), note for $\zeta=1$ that $0\neq \mu^{G,(K,\phi)\times(K,1)}_{(K,1)}$ if and only if there is some $s\in G$ such that $[K\cap{^sK},{\phi}|_{K\cap{^sK}}]_G=[K,1]_G$, and so there is some $g\in G$ such that $K={^g(K\cap ^sK)}$ and ${^{g}\phi}=1$, hence $\phi=1$. The converse is easy.
\end{proof}

For $(K,\phi),(L,\psi)\in \mathcal{M}^A_G$, set
\begin{align*}
    \gamma^G_{(K,\phi),(L,\psi)}&=|\{KsL\in K\backslash G/L\;|\;(K,\phi)\leq {^s(L,\psi)}\}|\\
    &=|\{sL\in G/L\;|\;(K,\phi)\leq {^s(L,\psi)}\}|.
\end{align*}
These numbers will be referred to as the \textit{marks} of $\MAG$, for they generalize the notion of marks of the Burnside ring.

\begin{lemma}\thlabel{markprops}
Let $(K,\phi),(L,\lambda)\in \MAG$. The following assertions hold:
\begin{enumerate}
    \item $\gamma^G_{(^gK,^g\phi),(^xL,^x\psi)}=\gamma^G_{(K,\phi),(L,\psi)}$ for all $g,x\in G$,
    \item $0\neq \gamma^G_{(K,\phi),(L,\psi)}$ if and only if $[K,\phi]_G\leq [L,\psi]_G$,
    \item $\gamma^G_{(K,\phi),(K,\phi)}=[N_G(K,\phi):K]$,
    \item  $\gamma^G_{(K,1),(L,1)}=|(G/L)^K|$,
    \item $\gamma^G_{(K,\phi),(L,\psi)}=\mu^{G,(K,\phi)\times (L,\psi^{-1})}_{(K,1)}$.
\end{enumerate}
\end{lemma}
\begin{proof}
All the assertions are known to \cite[Section 1]{Bolt1}, with the exception of Part (5): $[K,1]_G=[K\cap {^sL}, \phi\cdot {^s\psi^{-1}}]_G$ if and only if there is some $g\in G$ such that $(^gK,1)=(K\cap {^sL}, \phi\cdot {^s\psi^{-1}})$, but this implies that $g\in N_G(K)$, $K=K\cap {^sL}$ and $1=\phi\left({^s\psi^{-1}}|_K\right)$, and the latter happens if and only if $K\leq {^sL}$ and $\phi={^s\psi}|_{K}$. Therefore,
\begin{align*}
    \{KsL\in K\backslash G/L\;|\; [K,1]_G=[K\cap {^sL}, \phi\cdot {^s\psi^{-1}}]_G\} \\
    = \{KsL\in K\backslash G/L\; &|\; (K,\phi)\leq ({^sL}, {^s\psi})\}.
\end{align*}
\end{proof}

The map
\begin{align}
    \Phi^A_G:B^A(G)\longrightarrow \prod_{K\leq G}\Z Hom(K,A), [L,\psi]_G\mapsto \left( \sum_{\phi\in Hom(K,A)}\gamma^G_{(K,\phi),(L,\psi)}\phi\right)_{K\leq G}
\end{align}
is an injective ring homomorphism known as the \textit{mark morphism} \cite[Section 2]{Bolt2}. The group $G$ acts on $\prod_{K\leq G}\Z Hom(K,A)$ by $^g(\alpha_K)_{K\leq G}=(^g\alpha_{K^g})_{K\leq G}$, and the image of $\Phi^A_G$ lies in the subring $\gBAG=\left(\prod_{K\leq G}\Z Hom(K,A)\right)^G$, known as the \textit{ghost ring} of $\BAG$. A basis for $\gBAG$ is given by the elements $\widetilde{\phi}\in \prod_{K\leq G}\Z Hom(K,A)$ for $\phi\in Hom(K,A)$, whose $L$-th entry is
\begin{align}
    \widetilde{\phi}_{L}={^g\left(\sum_{n\in [N_G(K)/N_G(K,\phi)]}{^n\phi}\right)}=\sum_{n\in [N_G(K)/N_G(K,\phi)]}{^{gn}\phi}
\end{align}
if $L={^gK}$ for some $g\in G$, and it is zero if $L\neq_G K$. If $\psi\in Hom(L,A)$, then $\widetilde{\phi}=\widetilde{\psi}$ if and only if $[K,\phi]_G=[L,\psi]_G$, and so $\BAG$ and $\gBAG$ have the same rank.

We can examine the multiplication coefficients a bit further. For $(K,\phi),(L,\psi)\in \MAG$, we have that
\begin{align}
    \Phi^A_G([K,\phi]_G)\Phi^A_G([L,\psi]_G)=\sum_{[T,\tau]_G\in G\backslash \MAG}\mu^{G,(K,\phi)\times(L,\psi)}_{(T,\tau)} \Phi^A_G\left([T,\tau]_G\right),
\end{align}
so if $U\leq G$ and $\omega \in Hom(U,A)$, then looking at the coefficient of $\omega$ in the $U$-th entry of $\Phi_G^A\left([K,\phi]_G\cdot [L,\psi]_G\right)$, we get
\begin{align}
    \sum_{\lambda\in Hom(U,A)}\gamma^G_{(U,\omega \lambda^{-1}),(K,\phi)}\gamma^G_{(U,\lambda),(L,\psi)}=\sum_{[U,\omega]_G\leq [T,\tau]_G\in G\backslash \MAG}\mu^{G,(K,\phi)\times(L,\psi)}_{(T,\tau)}\gamma^G_{(U,\omega),(T,\tau)}.
\end{align}
as $\gamma^G_{(U,\omega),(T,\tau)}=0$ if $[U,\omega]_G\not\leq [T,\tau]_G$. Setting
\begin{align}
    \alpha^{G,(K,\phi)\times (L,\psi)}_{(U,\omega)}&=\sum_{\lambda\in Hom(U,A)}
    \gamma^G_{(U,\omega {\lambda}^{-1}),(K,\phi)}\gamma^G_{(U,\lambda),(L,\psi)},\\
    \beta^{G,(K,\phi)\times (L,\psi)}_{(U,\omega)}&=\sum_{[U,\omega]_G< [T,\tau]_G\in G\backslash \MAG}\mu^{G,(K,\phi)\times(L,\psi)}_{(T,\tau),
    }\gamma^G_{(U,\omega),(T,\tau)},
\end{align}
we get to the identity
\begin{align}
    \mu^{G,(K,\phi)\times(L,\psi)}_{(U,\omega)} &=\frac{|U|}{|N_G(U,\omega)|}\left(\alpha^{G,(K,\phi)\times (L,\psi)}_{(U,\omega)}- \beta^{G,(K,\phi)\times (L,\psi)}_{(U,\omega)}\right).
\end{align}

\section{Species isomorphisms}

\begin{deff}
Let $A$ be an abelian group, $G$ and $H$ be finite groups. A ring isomorphism $\Theta: B^A(G) \longrightarrow B^A(H)$ is a \textit{species isomorphism}  if $\Theta([K,\phi]_G)\in H\backslash\MAH$ for every $[K,\phi]_G\in G\backslash\MAG$.
\end{deff}

It is straightforward that a species isomorphism sends standard basis to standard basis bijectively. In the following, $\Theta:\BAG\longrightarrow \BAH$ is a species isomorphism.

\begin{lemma}\thlabel{mus}
Let $(K,\phi),(L,\psi), (T,\tau)\in \MAG$. If $\Theta([K,\phi]_G)=[R,\rho]_H$, $\Theta([L,\psi]_H)=[S,\sigma]_H$ and $\Theta([T,\tau]_G)=[U,\omega]_H$, then $\mu^{H,(R,\rho)\times (S,\sigma)}_{(U,\omega)}=\mu^{G,(K,\phi)\times (L,\psi)}_{(T,\tau)}$.
\end{lemma}
\begin{proof}
Since $\Theta$ is a ring isomorphism, the expressions
$$\Theta([K,\phi]_G\cdot [L,\psi]_G)=\sum_{[T,\tau]_G\in G\backslash \MAG}\mu^{G,(K,\phi)\times(L,\psi)}_{(T,\tau)}\cdot \Theta([(T,\tau)]_G)$$
$$\Theta([K,\phi]_G)\cdot\Theta( [L,\psi]_G)=\sum_{[U,\omega]_H\in H\backslash \MAH}\mu^{H,(R,\rho)\times(S,\sigma)}_{(U,\omega)}\cdot [U,\omega)]_H,$$
must be equal, so if $[U,\omega]_H=\Theta([T,\tau]_G)$, then $\mu^{H,(R,\rho)\times (S,\sigma)}_{(U,\omega)}=\mu^{G,(K,\phi)\times (L,\psi)}_{(T,\tau)}$.
\end{proof}

We now proceed to prove that species isomorphisms preserve the marks.

\begin{lemma}\thlabel{ideal11}
Let $(K,\phi)\in \MAG$. Then $\Z [K,\phi]_G$ is an ideal of $\BAG$ if and only if $[K,\phi]_G=[1,1]_G$.
\end{lemma}
\begin{proof}
Since $[L,\psi]_G\cdot [1,1]_G=[G:L] [1,1]_G$ for all $(L,\psi)\in \MAG$, it is clear that $\Z [1,1]_G$ is an ideal. Conversely, if $\Z [K,\phi]_G$ is an ideal, then for every $(L,\psi)$, $[L,\psi]_G\cdot [K,\phi]_G=n[K,\phi]_G$ for some $n\in \Z$, in particular, 
$$0\neq [G:K][1,1]_G=[1,1]_G\cdot [K,\phi]_G=n[K,\phi]_G,$$
and so $[K,\phi]_G=[1,1]_G$.
\end{proof}
 
\begin{cor}\thlabel{11to11}
$\Theta([1,1]_G)=[1,1]_H$. 
\end{cor}
\begin{proof}
$\Z\Theta([1,1]_G)$ is an ideal of $\BAH$, thus $\Theta([1,1]_G)=[1,1]_H$ by \thref{ideal11}.
\end{proof}
 
\begin{lemma}\thlabel{sgorders}
If $(K,\phi)\in \MAG$ and $\Theta([K,\phi]_G)=[R,\rho]_H$, then $|R|=|K|$.
\end{lemma}
\begin{proof}
Observe that
$$|H|\cdot [1,1]_H=[1,1]_H\cdot [1,1]_H=\Theta([1,1]_G\cdot [1,1]_G)=|G|\cdot\Theta([1,1]_G)=|G|\cdot[1,1]_H$$
and so $|H|=|G|$. Then
$$[H:R]\cdot [1,1]_H=[R,\rho]_H\cdot [1,1]_H=\Theta([K,\phi]_G\cdot [1,1]_G)=[G:K]\cdot [1,1]_H$$
and so $|R|=|K|$.
\end{proof}

\begin{lemma}\thlabel{1stEntr}
Let $\phi,\psi \in Hom(K,A)$. If $\Theta([K,\phi]_G)=[R,\rho]_H$ and $\Theta([K,\psi]_G)=[S,\sigma]_H$, then $R=_HS$.
\end{lemma}
\begin{proof}
It is clear that $0\neq \mu^{G,(K,\phi)\times (K,\psi)}_{(K,\phi\psi)}$. If $\Theta([K,\phi\psi])=[U,\omega]_H$, then $0\neq \mu^{G,(R,\rho)\times (S,\sigma)}_{(U,\omega)}$, hence there are $t,h\in H$ such that $({^h}U,{^h\omega})=(R\cap {^tS}, \rho\cdot {^t\sigma})$. But $|R|=|S|=|U|=|K|$ by \thref{sgorders}, and $R\cap {^tS}\subseteq R$ and $R\cap {^tS}\subseteq {^tS}$, thus $R={^tS}$.
\end{proof}

\begin{prop}\thlabel{FibtoBurn}
A species isomorphism between $\BAG$ and $\BAH$ can be restricted to an isomorphism of mark tables of between $B(G)$ and $B(H)$.
\end{prop}
\begin{proof}
Let $\Theta([K,1]_G)=[R,\rho]_H$. If $\sigma \in Hom(R,A)$, then $\Theta^{-1}([R,\sigma]_H)=[K,\zeta]_G$ for some $\zeta\in Hom(K,\sigma)$. By \thref{mus}, $0\neq \mu^{G,(K,1)\times(K,\zeta)}_{(K,\zeta)}= \mu^{H,(R,\rho)\times(R,\sigma)}_{(R,\sigma)}$ for every $\sigma$, thus $\rho=1$ by Part (4) of \thref{BurnsEls}.
\end{proof}

\begin{lemma}\thlabel{morphs1}
Let $\phi,\psi\in Hom(K,A)$, and let $\Theta([K,\phi]_G)=[R,\rho]_H$ and $\Theta([K,\psi]_G)=[R,\sigma]_H$. Then:
\begin{enumerate}
    \item $\Theta([K,\phi\psi]_G)=[R,\rho(^s\sigma)]_H$ for some $s\in N_H(R)$,
    \item $\Theta([K,\phi^{-1}]_G)=[R,\rho^{-1}]_H$.
\end{enumerate}
\end{lemma}
\begin{proof}
If $\Theta([K,\phi\psi]_G)=[R,\tau]_H$, then $0\neq \mu^{H,(R,\rho)\times(R,\sigma)}_{(R,\tau)}$, and so there are elements $s,t\in H$ such that $^tR=R\cap {^sR}$ and $^t\tau=\rho\cdot{^s\sigma}$, but this is implies that $s,t\in N_H(R)$ and $^t\tau=\rho(^s\sigma)$, hence $[R,\tau]_H=[R,{^t\tau}]_H=[R,\rho\cdot{^s\sigma}]_H$, proving Part (1). In particular, if $\psi=\phi^{-1}$, then by \thref{FibtoBurn} we have that $[R,1]_H=\Theta([K,\phi\phi^{-1}]_G)=[R,\rho(^s\sigma)]_H$ for some $s\in N_H(R)$, hence $^s\sigma=\rho^{-1}$, which proves Part (2).
\end{proof}

\begin{cor}\thlabel{gammas}
Let $(K,\phi),(L,\psi)\in \MAG$. If $\Theta([K,\phi]_G)=[R,\rho]_H$ and $\Theta([L,\psi]_G)=[S,\sigma]_H$, then $\gamma^H_{(R,\rho),(S,\sigma)}=\gamma^G_{(K,\phi),(L,\psi)}$.
\end{cor}
\begin{proof}
$\gamma^G_{(K,\phi),(L,\psi)}=\mu_{(K,1)}^{G,(K,\phi)\times (L,\psi^{-1})} =\mu_{(R,1)}^{H,(R,\rho)\times (S,\sigma^{-1})}=\gamma^H_{(R,\rho),(S,\sigma)}$.
\end{proof}

\begin{cor}\thlabel{lattiso}
The map $G\backslash \MAG \longrightarrow H\backslash \MAH, [K,\phi]_G\mapsto \Theta([K,\phi]_G)$ is an isomorphism of posets.
\end{cor}
\begin{proof}
Recalling that $[K,\phi]_G\leq [L,\psi]_G$ if and only if $0\neq \gamma^G_{(K,\phi),(L,\psi)}$, the result follows by \thref{gammas}.
\end{proof}
\begin{prop}\thlabel{bijBASES}
There is a bijection $\theta:\MAG \longrightarrow \MAH$ such that $\Theta([K,\phi]_G)=[\theta(K,\phi)]_H$. Furthermore, $\theta$ can be chosen to be of the form $\theta(K,\phi)=(\theta_{\mathcal{S}}(K),\theta_K(\phi))$ for some bijections $\theta_{\mathcal{S}}:\SG\longrightarrow \SH$ and $\theta_K:Hom(K,A)\longrightarrow Hom(\theta_{\mathcal{S}}(K),A)$ for $K\leq G$ satisfying the following:
\begin{enumerate}
    \item $\theta_K(1)=1$,
    \item $\forall \phi,\psi\in Hom(K,A)$, $\theta_K(\phi\psi)={^s\theta_K(\phi)}{^t\theta_K(\psi)}$ and $\theta_K(\phi^{-1})={^u\theta_K(\phi)^{-1}}$, for some $s,t,u \in N_H(\theta_{\mathcal{S}}(K))$,
    \item If $g\in G$, then $\theta_{^gK}(^g\phi)={^h}\theta_K(\phi)$ for some $h\in H$ such that $\theta_{\mathcal{S}}(^gK)={^h\theta_{\mathcal{S}}(K)}$. If $g\in N_G(K)$, then $h\in N_H(\theta_S(K))$.
\end{enumerate}
\end{prop}
\begin{proof}
Let $(K,\phi)\in \MAG$ and $\Theta([K,\phi]_G)=[R,\rho]_H$ for some $(R,\rho)\in \MAH$. By \thref{gammas}, we have that 
$$[N_G(K,\phi):K]=\gamma^G_{(K,\phi),(K,\phi)}=\gamma^{H}_{(R,\rho),(R,\rho)}=[N_H(R,\rho):R]$$
and since $|K|=|R|$ and $|G|=|H|$, we get that
$$|[K,\phi]_G|=[G:N_G(K,\phi)]=[H:N_H(R,\rho)]=|\Theta([K,\phi]_G)|,$$
thus we can construct $\theta$ by choosing a bijection between $[K,\phi]_G$ and $\Theta([K,\phi]_G)$ for every $[K,\phi]_G\in G\backslash \MAG$.

From \thref{FibtoBurn}, we know that $\Theta$ induces a bijection $\theta_{\mathcal{S}}:\SG\longrightarrow \SH$ preserving conjugacy classes, so we fix one such bijection. If $\phi\in Hom(K,A)$ and $\Theta([K,\phi]_G)=[\theta_{\mathcal{S}}(K),\rho]_H$, then $[N_G(K):N_G(K,\phi)]=[N_H(\theta_{\mathcal{S}}(K)):N_H(\theta_{\mathcal{S}}(K),\rho)]$, and choosing a bijection between $N_G(K)\cdot \phi$ and $N_H(\theta_{\mathcal{S}}(K))\cdot \rho$ for every $\phi \in [N_G(K)\backslash Hom(K,A)]$, we get a bijection $\theta_K:Hom(K,A) \longrightarrow Hom(\theta_{\mathcal{S}}(K),A)$. The last assertions are clear.
\end{proof}

By Parts (2) and (3) of \thref{bijBASES}, some of the maps $\theta_K$ are necessarily group isomorphisms.

\begin{cor}
If $N_G(K)=KC_G(K)$ or $N_H(\theta_{\mathcal{S}}(K))=\theta_{\mathcal{S}}(K)C_H(\theta_{\mathcal{S}}(K))$, or if $K$ is self-normalized, then $\theta_K$ is a group isomorphism. In particular, $\theta_G$ and $\theta_{Z(G)}$ are always group isomorphisms.
\end{cor}

If we have any bijection $\theta:\MAG\longrightarrow \MAH$ preserving the marks, then we still have an isomorphism of posets between $G\backslash \MAG$ and $H\backslash \MAH$ and $|G|=|H|$, and if these sets are finite, they can be given total orders such that
\begin{align}
    \left[\gamma^G_{(K,\phi),(L,\psi)}\right]_{[K,\phi]_G, [L,\psi]_G\in G\backslash\MAG}=\left[\gamma^H_{(R,\rho),(S,\sigma)}\right]_{[R,\rho]_H, [S,\sigma]_H\in H\backslash\MAH}
\end{align}
as $|G\backslash\MAG|\times |G\backslash\MAG|$-matrices. When $A$ has trivial $|G|$-torsion, $\theta$ induces an isomorphism of mark tables and $B(G)$ and $B(H)$ can be embeded as the same subring in a product of copies of $\Z$ \cite[Section 2]{RagVal1}. For more general $A$, we have no direct way to identify $\widetilde{B^A}(G)$ and $\widetilde{B^A}(H)$ so we could compare the images of the mark morphisms. However, when $\BAG$ and $\BAH$ are species-isomorphic, we can prove that their ghost rings are isomorphic too in a way that the images of the mark morphisms agree.

\begin{theorem}
If $\Theta:\BAG\longrightarrow \BAH$ is a species isomorphism, then there is a unique ring isomorphism $\widetilde{\Theta}:\widetilde{B^A}(G)\longrightarrow\widetilde{B^A}(H)$ such that the diagram
\begin{align}
    \xymatrix{
    \BAG\ar[rr]^{\Theta}\ar[d]_-{\Phi^A_G} &&\BAH\ar[d]^{\Phi^A_H}\\
    \widetilde{B^A}(G)\ar[rr]_-{\widetilde{\Theta}} &&\widetilde{B^A}(H)
}
\end{align}
commutes.
\end{theorem}
\begin{proof}
Let $\theta_{\mathcal{S}}:\SG\longrightarrow \SH$ and $\theta_K:Hom(K,A)\longrightarrow Hom(\theta_{\mathcal{S}}(K),A)$ be as in \thref{bijBASES}, and set $\theta(K,\phi)=(\theta_{\mathcal{S}}(K),\theta_K(\phi))$. Then $\widetilde{\Theta}:\widetilde{B^A}(G)\longrightarrow \widetilde{B^A}(H),\;\widetilde{\phi}\mapsto \widetilde{\theta_K(\phi)}$ is an isomorphism of abelian groups. Since multiplication in $\widetilde{B^A}(G)$ is given entry-to-entry, it is enough to verify that $\widetilde{\Theta}$ respects multiplication for elements $\widetilde{\phi}$ and $\widetilde{\psi}$ with $\phi,\psi\in Hom(K,A)$. Noting that the $K$-th entry of $\Phi^A_G([K,\phi]_G)$ equals $\gamma^G_{(K,\phi),(K,\phi)}\widetilde{\phi}_K$, then for $\tau\in Hom(K,\psi)$, we have that
\begin{align}
    \textit{coeff}_{\widetilde{\tau}}\left(\widetilde{\phi}\widetilde{\psi}\right)=\frac{\alpha^{G,(K,\phi)\times (K,\psi)}_{(K,\tau)}}{\gamma^G_{(K,\phi),(K,\phi)}\gamma^G_{(K,\psi),(K,\psi)}},
\end{align}
and since $0=\beta^{G,(K,\phi)\times (K,\psi)}_{(K,\tau)}=\beta^{H,\theta(K,\phi)\times \theta(K,\psi)}_{\theta(K,\tau)}$ and $\mu^{G,(K,\phi)\times (K,\psi)}_{(K,\tau)}=\mu^{H,\theta(K,\phi)\times \theta(K,\psi)}_{\theta(K,\tau)}$, then by Equation (11), we get
\begin{align*}
    \textit{coeff}_{\widetilde{\rho}}\left(\widetilde{\theta_K(\phi)}\widetilde{\theta_K(\psi)}\right)&=\frac{\alpha^{H,\theta(K,\phi)\times \theta(K,\psi)}_{(\theta_{\mathcal{S}}(K),\rho)}}{\gamma^H_{\theta(K,\phi),\theta(K,\phi)}\gamma^H_{\theta(K,\psi),\theta(K,\psi)}}\\
    &=\frac{\alpha^{G,(K,\phi)\times (K,\psi)}_{(K,\tau)}}{\gamma^G_{(K,\phi),(K,\phi)}\gamma^G_{(K,\psi),(K,\psi)}}=\textit{coeff}_{\widetilde{\tau}}\left(\widetilde{\phi}\widetilde{\psi}\right),
\end{align*}
if $\rho=\theta_K(\tau)$. Hence, $\widetilde{\Theta}\left(\widetilde{\phi}\widetilde{\psi}\right)=\widetilde{\Theta}\left(\widetilde{\phi}\right)\widetilde{\Theta}\left(\widetilde{\psi}\right)$ for all $\phi\in Hom(K,A)$, $\psi\in Hom(L,A)$. For every $(K,\phi)\in \MAG$ and $R=\theta_{\mathcal{S}}(T)\leq H$, we have that
\begin{align*}
    \left(\widetilde{\Theta}\Phi^A_G\left([K,\phi]_G\right)\right)_R &=\sum_{
    \tau\in [N_G(T)\backslash Hom(T,A)]}\gamma^G_{(T,\tau),(K,\phi)}\widetilde{\Theta}\left(\widetilde{\lambda}\right)_R\\
    &=\sum_{
    \tau\in [N_G(T)\backslash Hom(T,A)]}\gamma^H_{\theta(T,\tau),\theta(K,\phi)}\widetilde{\theta_T(\lambda)}_R\\
    &=\sum_{
    \rho\in [N_H(R)\backslash Hom(R,A)]}\gamma^H_{(R,\rho),\theta(K,\phi)}\widetilde{\rho}_R =\left(\Phi^A_H\Theta([K,\phi]_G)\right)_R,
\end{align*}
thus $\widetilde{\Theta}\Phi^A_G=\Phi^A_H\Theta$. Finally, if $\widetilde{\Theta}':\widetilde{B^A}(G)\longrightarrow \widetilde{B^A}(H)$ is another isomorphism making Diagram (13) commute, then tensoring by $\Q$ we get that 
$$\Q \otimes \widetilde{\Theta}'= \left(\Q\otimes\Phi^A_H \right) \circ \left(\Q\otimes\Theta \right)\circ \left(\Q\otimes\Phi^A_G \right)^{-1}=\Q \otimes \widetilde{\Theta},$$
hence $\widetilde{\Theta}'=\widetilde{\Theta}$.
\end{proof}

There is a partial converse to this result in the following theorem, which states three characterizations of species isomorphisms.

\begin{theorem}\thlabel{MarkIso}
Let $A$ be an abelian group and $G$ and $H$ be finite groups. The following statements are equivalent:
\begin{enumerate}
    \item There is a species isomorphism from $\BAG$ to $\BAH$,
    \item There is a bijection $\theta:\MAG\longrightarrow \MAH$ satisfying the following conditions:
    \begin{enumerate}
    \item $\gamma^{H}_{\theta(K,\phi),\theta(L,\psi)}=\gamma^{G}_{(K,\phi),(L,\psi)}$ for all $(K,\phi),(L,\psi)\in \MAG$,
    \item $\mu^{H,\theta(K,\phi)\times \theta(L,\psi)}_{\theta(T,\tau)}=\mu^{G,(K,\phi)\times (L,\psi)}_{(T,\tau)}$ for all $(K,\phi),(L,\psi),(T,\tau)\in \MAG$,
    \end{enumerate}
    \item There are bijections $\theta_{\mathcal{S}}:\SG\longrightarrow \SH$ and     $\theta_K:Hom(K,A)\longrightarrow Hom(\theta_{\mathcal{S}}(K),A)$ for $K\leq G$ such that
    $$\mu^{H,(\theta_{\mathcal{S}}(K),\theta_K(\phi))\times (\theta_{\mathcal{S}}(L),\theta_L(\psi))}_{(\theta_{\mathcal{S}}(T),\theta_T(\tau))}=\mu^{G,(K,\phi)\times (L,\psi)}_{(T,\tau)}$$ 
    for all $(K,\phi),(L,\psi),(T,\tau)\in \MAG$,
    \item There are bijections $\theta_{\mathcal{S}}:\SG\longrightarrow \SH$ and $\theta_K:Hom(K,A)\longrightarrow Hom(\theta_{\mathcal{S}}(K),A)$ for $K\leq G$ such that:
    \begin{enumerate}
    \item  $\gamma^{H}_{(\theta_{\mathcal{S}}(K),\theta_K(\phi)),(\theta_{\mathcal{S}}(L),\theta_L(\psi))}=\gamma^{G}_{(K,\phi),(L,\psi)}$ for all $(K,\phi),(L,\psi)\in \MAG$,
    \item The map $\widetilde{\Theta}:\widetilde{B^A}(G)\longrightarrow \widetilde{B^A}(H)$ given by $\widetilde{\phi}\mapsto \widetilde{\theta_K(\phi)}$ for $\phi\in Hom(K,A)$ is a ring isomorphism.
    \end{enumerate}
\end{enumerate}
\end{theorem}
\begin{proof}
(1) $\Longrightarrow$ (3): If $\Theta:\BAG\longrightarrow \BAH$ is a species isomorphism, then we can take bijections $\theta_{\mathcal{S}}:\SG\longrightarrow \SH$ and $\theta_K:Hom(K,A)\longrightarrow Hom(\theta_{\mathcal{S}}(K),A)$ for $K\leq G$ as in \thref{bijBASES}, which satisfy the condition by \thref{mus}.

(3) $\Longrightarrow$ (4): Note that $0\neq \mu^{G,(K,1)\times (K,\theta_K^{-1}(\zeta))}_{(K,\theta_K^{-1}(\zeta))}=\mu^{H,(\theta_{\mathcal{S}}(K),\theta_K(1))\times (\theta_{\mathcal{S}}(K),\zeta)}_{(\theta_{\mathcal{S}}(K),\zeta)}$ for every $\zeta\in Hom(\theta_{\mathcal{S}}(K),A)$, hence $\theta_K(1)=1$ by \thref{BurnsEls}. Write $\theta(K,\phi)=(\theta_{\mathcal{S}}(K),\theta_K(\phi))$ for simplicity. For $\phi,\psi\in Hom(K,A)$, we have that $0\neq \mu^{G,(K,\phi)\times (K,\psi)}_{(K,\phi\psi)}=\mu^{H,\theta(K,\phi)\times \theta(K,\psi)}_{\theta(K,\phi\psi)}$, thus there are $s,t\in N_H(\theta_{\mathcal{S}}(K))$ such that $\theta_K(\phi\psi)={^s\theta_K(\phi)}{^t\theta_K(\psi)}$. In particular, $\theta_K(\phi^{-1})={^t\theta_K(\phi)}^{-1}$ for some $t\in N_H(\theta_{\mathcal{S}}(K))$. Therefore,
\begin{align*}
    \gamma^{G}_{(K,\phi),(L,\psi)}&=\mu^{G,(K,\phi)\times (L,\psi^{-1})}_{(K,1)}\\
    &=\mu^{H,(\theta_{\mathcal{S}}(K),\theta_K(\phi))\times (\theta_{\mathcal{S}}(L),\theta_L(\psi^{-1}))}_{(\theta_{\mathcal{S}}(K),\theta_K(1))}\\
    &=\mu^{H,(\theta_{\mathcal{S}}(K),\theta_K(\phi))\times (\theta_{\mathcal{S}}(L),\theta_L(\psi)^{-1})}_{(\theta_{\mathcal{S}}(K),1)}\\
    &=\gamma^{H}_{\theta(K,\phi),\theta(L,\psi)}
\end{align*}
for all $(K,\phi),(L,\psi)\in \MAG$. It follows that the maps $\theta_K$ preserve orbits, and the map $\widetilde{\Theta}:\gBAG \longrightarrow \gBAH$, given by $\widetilde{\phi}\mapsto \widetilde{\theta_K(\phi)}$ for $\phi\in Hom(K,A)$, $K\leq G$, is a group isomorphism that sends $(1)_{K\leq G}$ to $(1)_{R\leq H}$. Since $0=\beta^{G,(K,\phi)\times (K,\psi)}_{(K,\lambda)}=\beta^{H,\theta(K,\phi)\times \theta(K,\psi)}_{\theta(K,\lambda)}$, and by Equations (11) and (14),
\begin{align*}
    \textit{coeff}_{\widetilde{\tau}}\left(\widetilde{\phi}\widetilde{\psi}\right)&=\frac{\alpha^{G,(K,\phi)\times (K,\psi)}_{(K,\tau)}}{\gamma^G_{(K,\phi),(K,\phi)}\gamma^G_{(K,\psi),(K,\psi)}}\\
    &=\frac{\alpha^{H,\theta(K,\phi)\times \theta(K,\psi)}_{\theta(K,\tau)}}{\gamma^H_{\theta(K,\phi),\theta(K,\phi)}\gamma^H_{\theta(K,\psi),\theta(K,\psi)}}\\
    &=\textit{coeff}_{\widetilde{\theta_K(\tau)}}\left(\widetilde{\theta_K(\phi)}\widetilde{\theta_K(\psi)}\right)
\end{align*}
for all $\phi,\psi, \tau \in Hom(K,A)$, which proves that $\widetilde{\Theta}$ is multiplicative.

(4) $\Longrightarrow$ (2): Let $\theta(K,\phi)= (\theta_{\mathcal{S}}(K),\theta_K(\phi))$. By Condition (a), it is easy to see that $\widetilde{\Theta}\left(\Phi^A_G([K,\phi]_G)\right)=\Phi^A_H\left([\theta(K,\phi)]_H\right)$. Since $\widetilde{\Theta}$ is a ring isomorphism, we have that
\begin{align*}
    \alpha^{G,(K,\phi)\times (L,\psi)}_{(T,\tau)}&=   \textit{coeff}_{\widetilde{\tau}}\left(\Phi^A_G([K,\phi]_G)\Phi^A_G([L,\psi]_G)\right)\\
    &=\textit{coeff}_{\widetilde{\Theta}\left(\widetilde{\tau}\right)}\left(\Phi^A_H([\theta(K,\phi)]_H)\Phi^A_H([\theta(L,\psi)]_H)\right)=\alpha^{H,\theta(K,\phi)\times \theta(L,\psi)}_{\theta(T,\tau)}
\end{align*}
for all $(K,\phi),(L,\psi),(T,\tau)\in \MAG$. Recalling Equation (11), we now prove by reversed induction on $|T|$ that $\mu^{G,(K,\phi)\times (L,\psi)}_{(T,\tau)}=\mu^{H,\theta(K,\phi)\times \theta(L,\psi)}_{\theta(T,\tau)}$ for all $(K,\phi),(L,\psi),(T,\tau)\in \MAG$. If $|T|=|G|$, then $T=G$, $\theta_S(G)=H$, and $0=\beta^{G,(K,\phi)\times (L,\psi)}_{(G,\tau)}=\beta^{H,\theta(K,\phi)\times \theta(L,\psi)}_{(H,\theta_G(\tau))}$, hence
$$\mu^{G,(K,\phi)\times (K,\psi)}_{(G,\tau)}=\alpha^{G,(K,\phi)\times (L,\psi)}_{(G,\tau)}=\alpha^{H,\theta(K,\phi)\times \theta(L,\psi)}_{(H,\theta_G(\tau))}=\mu^{H,\theta(K,\phi)\times \theta(L,\psi)}_{(H,\theta_G(\tau))}.$$
Assume now that $|T|<|G|$ and $\mu^{G,(K,\phi)\times (L,\psi)}_{(U,\omega)}=\mu^{H,\theta(K,\phi)\times \theta(L,\psi)}_{\theta(U,\omega)}$ for all $(U,\omega)$ such that $|U|>|T|$. By Condition (a), $[U,\omega]_G\mapsto [\theta(U,\omega)]_H$ is an order-preserving bijection, and by the induction hypothesis and Condition (a), $\beta^{G,(K,\phi)\times (L,\psi)}_{(T,\tau)}=\beta^{H,\theta(K,\phi)\times \theta(L,\psi)}_{\theta(T,\tau)}$. Hence,
\begin{align*}
    \mu^{G,(K,\phi)\times(L,\psi)}_{(T,\tau)} &=\frac{1}{\gamma^G_{(T,\tau),(T,\tau)}}\left(\alpha^{G,(K,\phi)\times (L,\psi)}_{(T,\tau)}- \beta^{G,(K,\phi)\times (L,\psi)}_{(T,\tau)}\right)\\
    &=\frac{1}{\gamma^H_{\theta(T,\tau),\theta(T,\tau)}}\left(\alpha^{H,\theta(K,\phi)\times \theta(L,\psi)}_{\theta(T,\tau)}- \beta^{H,\theta(K,\phi)\times \theta(L,\psi)}_{\theta(T,\tau)}\right)=\mu^{H,\theta(K,\phi)\times \theta(L,\psi)}_{\theta(T,\tau)}.
\end{align*}

(2) $\Longrightarrow$ (1): Setting $\Theta([K,\phi]_G)=[\theta(K,\phi)]_G$ for $(K,\phi)\in \MAG$ defines an isomorphism of abelian groups $\Theta:\BAG\longrightarrow \BAH$. It is a ring isomorphism since, for all $(K,\phi),(L,\psi)\in \MAG$, we have that
\begin{align*}
    \Theta\left([K,\phi]_G\cdot [L,\psi]_G\right)&=\sum_{[T,\tau]_G\in G\backslash \MAG}\mu^{G,(K,\phi)\times(L,\psi)}_{(T,\tau)}\cdot [\theta(T,\tau)]_H\\
    &=\sum_{[T,\tau]_G\in G\backslash \MAG}\mu^{H,\theta(K,\phi)\times\theta(L,\psi)}_{\theta(T,\tau)}\cdot [\theta(T,\tau)]_H\\
    &=\sum_{[U,\omega]_H\in H\backslash \MAH}\mu^{H,\theta(K,\phi)\times\theta(L,\psi)}_{(U,\omega)}\cdot [(U,\omega)]_H\\
    &=\Theta([K,\phi]_G)\cdot \Theta([L,\psi]_G),
\end{align*}
and $\Theta([G,1]_G)=[H,1]_H$ because this is the only element acting as the identity on every element of $H\backslash\MAH$.
\end{proof}

\subsection*{Acknowledgments}

This work was developed during a postdoctoral fellowship at the PCCM UNAM-UMSNH, funded by a CONACyT EPM grant. I want to thank Professor Alberto Gerardo Raggi-Cárdenas, who suggested the study of isomorphisms preserving the standard bases as an adequate approach to the isomorphism problem of fibered Burnside rings. Special thanks to Professor Robert Boltje for introducing me to the study of monomial representation rings and their mark morphisms.

\end{document}